\documentclass[11pt]{article}
\usepackage{graphicx}
\usepackage{amsthm, amsmath, amssymb, tikz, bm}
\usepackage{mathtools, intcalc}
\usepackage{xifthen}
\usepackage{comment}
\usepackage[colorlinks=true,
linkcolor=blue,citecolor=blue,
urlcolor=blue]{hyperref}

\usepackage[margin=1.2in]{geometry} 

\usepackage[shortlabels]{enumitem}
\usepackage{todonotes}
\usetikzlibrary{calc,shapes, backgrounds}
\allowdisplaybreaks

\usepackage[utf8]{inputenc}
\usepackage[T1,T2A]{fontenc}
\usepackage[russian,english]{babel}

\newtheorem{lemma}{Lemma}
\newtheorem{theorem}{Theorem}

\newtheorem{proposition}{Proposition}

\newtheorem{fact}{Fact}

\usepackage[normalem]{ulem}

\title{Graph Configurations and Independent Bondage Numbers of Planar Graphs}

\author{
  E.G.K.M.Gamlath\thanks{University of Mississippi, University, MS 38677, (\texttt{egkmgamlath@gmail.com})}
  \and
  Bing Wei\footnotemark[1]
  \and
  Talmage James Reid\footnotemark[1]
}

\begin{document}
\maketitle
 
\begin{abstract}
The independent domination number of a finite graph G is the minimum cardinality of an independent dominating set of vertices. The independent bondage number of G is the minimum cardinality of a set of edges whose deletion results in a graph with a larger independent domination number than that of G. In this research, we enhance the existing upper bound on the independent bondage number of a planar graph with a minimum degree of at least three by identifying specific configurations within such planar graphs.
\end{abstract}

\section{Introduction and Main Results}
\label{introductiondefs}

Domination and bondage numbers are widely studied graph parameters, with various types under exploration. This paper focuses on the independent domination number and the independent bondage number, as surveyed in \cite{goddard2013independent} and \cite{xu2013bondage}, respectively.

The bondage number, introduced in 1983 by Bauer, Harary, Nieminen, and Suffel \cite{bauer1983domination}, is defined as the domination line-stability number of a graph. In 1990, Fink et al. \cite{fink1990bondage} highlighted its usefulness in measuring the vulnerability of an interconnection network under link failure. They conjectured that the bondage number of a graph is at most its maximum degree plus one. This paper enhances the established upper bound for the independent bondage number of planar graphs with a minimum degree of at least three. The improvement is achieved by incorporating specific configurations within these planar graphs.

Before presenting our main results, we provide definitions and terminology. We consider only finite, undirected, and simple planar graphs throughout the paper. Let $G$ be a graph with vertex set $V(G)$ and edge set $E(G)$. Let $v$ be a vertex and $f$ be a face of $G$, then
the degree of a vertex $v$, and the degree of a face $f$ are the number of vertices that are incident with them. These quantities are denoted by $d(v)$ and $d(f)$, respectively. For a vertex $v$ in $G$, the open neighborhood of $v$ is the vertices set of $G$ that is incident with $v$, and is denoted by $N_G(v)$. Thus, $d(v)=|N(v)|.$ The closed neighborhood of $G$ is denoted by $N[v]$, and $N[v]=N(v)\cup \{v\}$. The minimum and maximum degree of $G$ are denoted by $\delta(G)$ and $\Delta(G)$, respectively.

Let $j \in \mathbb{N}$. Then, a $j$-vertex (face), $j^+$-vertex (face), and $j^-$-vertex (face) refer to a vertex (face) that has a degree exactly $j$, a vertex that has a degree at least $j$, and a vertex that has a degree at most $j$, respectively. The set of faces that are incident with $v$ is denoted by $F(v)$. Let $N_j(v)$, $N_{j^+}(v)$, and $N_{j^-}(v)$, respectively, denote the set of neighbors of $v$ with degree exactly $j$, at least $j$, and at most $j$. Let $F_j(v)$, $F_{j^+}(v)$, and $F_{j^-}(v)$, respectively, denote the set of faces incident with $v$ with exactly degree $j$, at least $j$, and at most $j$. An edge with the sum of the degrees of its end vertices at most $j$ is called a $j^-$-{\it edge}. 

Define an $n$-{\it fan} with hub $v$ and rim $\{u_1,u_2,\ldots, u_n\}$ to be a set of vertices $\{v,u_1,u_2,\ldots, u_n\}$ that is maximal with respect to the property that  $(v,u_i)$ are edges for $i\in \{1,2,\dots, n\}$ and $(u_i,u_{i+1})$  are edges for $i\in \{1,2,\dots,n-1\}$. In addition, if  $(u_1,u_n) \in E(G)$, then we say that this set of vertices is an n-{\it wheel}. For example, Figure~\ref{nineconfig} $e(i)$ contains a $9$-fan with hub $v$ and rim $\{u_1,u_2,\ldots,u_9\}$. Note that this is a loose definition of the terms fan and wheel as we are only naming the vertices of the subgraph and we do not demand that the vertices on the rim of the fan (wheel) have degree three.

A \textbf{dominating set} of $G$ is defined as a set of vertices $D \subseteq V(G)$ such that each vertex of $G$ is either in $D$ or incident with a vertex in $D$. The \textbf{domination number} $\gamma(G)$ of $G$ is the cardinality of the smallest dominating set. The \textbf{bondage number} of a graph $G$, denoted by $b(G)$, is the minimum cardinality among all edge sets $B \subseteq E(G)$ such that $\gamma(G-B) > \gamma(G)$. A set of vertices of $G$ that is both independent and dominating is called an \textbf{independent dominating set}. The minimum cardinality among all independent dominating sets is called the \textbf{independent domination number}, denoted by $\gamma^i_{\ }(G)$. The \textbf{independent bondage number} of a graph $G$ is denoted by $b_i(G)$ and is defined as the minimum cardinality among all edge sets $B \subseteq E(G)$ such that $\gamma^i_{\ }(G-B) > \gamma^i_{\ }(G)$.
\vspace{0.3cm}

Our main result is the following theorem. 
\begin{theorem}
\label{IBNconfigurationproof}
If $G$ is a planar graph with $\delta(G) \geq 3$, then $b_i(G) \leq 8$.
\end{theorem}

We find some configurations of a planar graph with minimum degree at least $3$ to prove the Theorem \ref{IBNconfigurationproof}. We prove that,

\begin{theorem}
\label{maintheorem}
A planar graph  $G$ with $\delta(G) \geq 3$ contains an edge $e$ as in cases (a), (b), or (c), or a vertex $v$ and  an independent subset $I$ of the neighbor set 
$N(v)=\{u_i:i\in \{1,2,\dots,d(v)\}\}$ as given in cases (d) through (h).

(see the appropriate subcase of Figure~\ref{nineconfig} for   the arrangement of faces in each of configurations $(d)$ through $ (h)$).  
\begin{itemize} 
    \item[(a)] The edge $e$ is 
    an $11^-$-edge incident with two 3-faces. 
    \item[(b)]  The edge $e$ is a $10^-$- edge incident with a 3-face. 
    \item[(c)]  The edge $e$ is a $9^-$- edge. 
    \item[(d)]  $d(v)\in \{9,10\}$  and  $I$ contains  either
    \begin{itemize}
        \item[(i)] at least $\left\lfloor\frac{d(v)}{2}\right\rfloor - 1$ vertices of degree $3$ and an additional vertex of degree at most $5$ (for example, in Figure \ref{nineconfig} $(d)$, $d(v)=10$, so $d(u_i)=3$ for $i=1,3,5,7$, and $d(u_9)\leq 5$), or 
        \item[(ii)] at least $\left\lfloor\frac{d(v)}{2}\right\rfloor - 2$ vertices of degree $3$ and two additional vertices of degree at most $4$ (for example, in Figure \ref{nineconfig} $(d)$, $d(v)=10$, so $d(u_i)=3$ for $i=1,3,5$, and $d(u_i)\leq 4$ for $i=7$ and $9$).
.
    \end{itemize} 
    \item[(e)]  $d(v)=9$ and 
case (i) or (ii) holds.
\begin{itemize} 
\item[(i)] The vertex $v$ is incident with exactly one $4^+$-face and $I$ contains either
\begin{itemize}
\item[($\alpha$)]   four vertices $u_i$ for $i  \in \{2,3,\dots,8\}$ with three of the vertices of degree three and an additional vertex of degree at most five ((for example, in Figure \ref{nineconfig} $(e)$, $d(u_i)=3$ for $i=2,4,6$ and $d(u_8)\leq 5$), or
\item[($\beta$)]  
the vertex $u_1$ of degree three, $u_9$ of degree at most seven 
 (or interchange $u_1$ and $u_9$ above), and  $d(u_i)=3$ for $i=3,5,7 $, or

\item[($\delta$)]    $u_1$, $u_8$, and two $u_i$ for $i\in \{3,4,\dots, 6\}$ all of degree three   (or $u_2$, $u_9$, and two $u_i$ for $i\in \{4,5,6,7\}$ all of degree three).
\end{itemize}
   
 \item[(ii)] The vertex $v$ is incident seven $3$-faces and two $4^+$-faces with edge 
 $(v,u_9)$ on the two $4^+-$faces, 
    and $I$ contains  the vertex $u_1$ (or $u_8$) of degree three, three vertices $u_i$ for $i\in \{3,\dots, 7\}$ (or $i\in \{2,\dots, 6\}$ ) of degree three  and $u_9$ of degree at most six.

    
\end{itemize}
        \item[(f)] $d(v)=8$ and $v$ is 
        incident with six 3-faces and two $4^+$-faces as shown in Figure \ref{nineconfig}, $(f)$, such that $I$ contains vertices $u_i$ for $i=1,2,4$, and $5$ of degree three, and vertex $u_7$ of degree at most five
        (or take a counterclockwise traverse of Figure~\ref{nineconfig} (f)  obtained by applying the  permutation $(12)(38)(47)(56)$ to the indices above)
           
    \item[(g)] $d(v)=8$ and $v$ is incident with five 3-faces and three 
    $4^+$-faces as shown in Figure \ref{nineconfig} $(g)$, and I contains vertices $u_i$ for  $i=2,4,5,7$, and $8$ of degree three.
     
    \item[(h)] $d(v)=10$ and $v$ is incident with nine $3$-faces, one $4^+$-face (Figure \ref{nineconfig} $(h)$), and $I$ contains five vertices of degree three consisting of four $u_i$ for  $i\in \{2,3,\dots,9\}$, and either $u_1$ or $u_{10}$ ($d(u_i)=3$ for $i=1,3,5,7,9$ or $i=2,4,6,8,10$).    
\end{itemize}
\end{theorem}
   \begin{center}
   \begin{figure}[h]

\begin{tikzpicture}[scale=0.6]

    \node[circle,scale=0.5, draw=black, fill=black, minimum size=0.1mm](1) at (0,0){};
\node[below of=1,yshift=28,xshift= -.5,label=:\small{$v$}]{};
    \path (1) ++(0:2) node[circle,scale=0.5,draw=black, fill=black, minimum size=0.1mm, label=right:\small{$u_1$}](2) {};
    \path (1) ++(36:2) node[circle,scale=0.5,draw=black, fill=black, label=right:\small{$u_{10}$}](3){}; 
    \path (1) ++(72:2) node[circle,scale=0.5,draw=black, fill=black](4){}; 
     \node[below of=4,yshift=28,xshift= 5,label=:\small{$u_9$}]{};
    \path (1) ++(108:2) node[circle,scale=0.5,draw=black, fill=black, label=above:\small{$u_8$}](5){}; 
   
    \path (1) ++(144:2) node[circle,scale=0.5,draw=black, fill=black, label=left:\small{$u_7$}](6){}; 
    \path (1) ++(180:2) node[circle,scale=0.5,draw=black, fill=black, label=left:\small{$u_6$}](7){}; 
    \path (1) ++(216:2) node[circle,scale=0.5,draw=black, fill=black, label=left:\small{$u_5$}](8){};
    \path (1) ++(252:2) node[circle,scale=0.5,draw=black, fill=black, label=below:\small{$u_4$}](9){};

    \path (4) ++(108:2) node[circle,scale=0.5,draw=black, fill=black, label=left:\small{$x$}](12){}; 
    \path (4) ++(36:2) node[circle,scale=0.5,draw=black, fill=black, label=right:\small{$y$}](13){}; 

\node[below of=9,label=below:$(d)$,xshift=10]{};
    
    \path (1) ++(288:2) node[circle,scale=0.5,draw=black, fill=black, label=below:\small{$u_3$}](10){};
    \path (1) ++(324:2) node[circle,scale=0.5,draw=black, fill=black, label=right:\small{$u_2$}](11){};

    \path [-,ultra thin] (1) edge node[left] {} (2);
    \path [-,ultra thin] (1) edge node[left] {} (3);
    \path [-,ultra thin] (1) edge node[left] {} (4);
    \path [-,ultra thin] (1) edge node[left] {} (5);
    \path [-,ultra thin] (1) edge node[left] {} (6);
    \path [-,ultra thin] (1) edge node[left] {} (7);
    \path [-,ultra thin] (1) edge node[left] {} (8);
     \path [-,ultra thin] (1) edge node[left] {} (9);
     \path [-,ultra thin] (1) edge node[left] {} (10);
     \path [-,ultra thin] (1) edge node[left] {} (11);

    \path [-,ultra thin] (3) edge node[left] {} (4);
     \path [-,ultra thin] (4) edge node[left] {} (5);
    \path [-,ultra thin] (5) edge node[left] {} (6);
    \path [-,ultra thin] (6) edge node[left] {} (7);
     \path [-,ultra thin] (7) edge node[left] {} (8);
     \path [-,ultra thin] (8) edge node[left] {} (9);
     \path [-,ultra thin] (9) edge node[left] {} (10);
     \path [-,ultra thin] (10) edge node[left] {} (11);
     \path [-,ultra thin] (2) edge node[left] {} (11);
    \path [-,ultra thin] (2) edge node[left] {} (3);

    \path [-,ultra thin] (4) edge node[left] {} (12);
    \path [-,ultra thin] (4) edge node[left] {} (13);

    \begin{scope}[xshift=270]

    \node[circle, draw=black, fill=black, minimum size=0.1mm, scale=0.5](1) at (0,0){};
    \node[below of=1,yshift=28,xshift= -2,label=:\small{$v$}]{};
   
    \path (1) ++(-1:2) node[circle,draw=black, fill=black, minimum size=0.1mm, label=below right:\small{$u_1$},scale=0.5](2) {};

\node[right of=1,yshift=-5.5,label=:\tiny{$f_1$}]{};
    
     \path (1) ++(-3:2.5) node[circle,scale=0.5,xshift=25](20) {};
    \path (1) ++(37:2) node[circle,draw=black,scale=0.5, fill=black, label=above:\small{$u_9$}](3){}; 
    \path (1) ++(77:2) node[circle,draw=black, fill=black,scale=0.5](4){}; 
\node[below of=4,yshift=28,xshift= 5,label=:\small{$u_8$}]{};

  \path (4) ++(108:2) node[circle,scale=0.5,draw=black, fill=black, label=left:\small{$x$}](15){}; 
    \path (4) ++(36:2) node[circle,scale=0.5,draw=black, fill=black, label=right:\small{$y$}](14){};

    \path (1) ++(117:2) node[circle,draw=black, fill=black,scale=0.5, label=left:\small{$u_7$}](5){}; 
    \path (1) ++(157:2) node[circle,draw=black, fill=black,scale=0.5, label=left:\small{$u_6$}](6){}; 
    \path (1) ++(197:2) node[circle,draw=black, fill=black,scale=0.5, label=left:\small{$u_5$}](7){}; 
    \path (1) ++(237:2) node[circle,draw=black, fill=black,scale=0.5, label=left:\small{$u_4$}](8){};
    \path (1) ++(277:2) node[circle,draw=black, fill=black,scale=0.5, label=below:\small{$u_3$}](9){};
    \node[below of=9,label=below:$(e(i))$]{};
    \path (1) ++(317:2) node[circle,scale=0.5,draw=black, fill=black, label=below:\small{$u_2$}](10) {};
    \path (1) ++(17:3.8) node[circle,scale=0.5,draw=black, fill=black](11) {};

    \path [-,ultra thin] (1) edge node[left] {} (2);
    \path [-,ultra thin] (1) edge node[left] {} (3);


    \path [-,ultra thin] (1) edge node[left] {} (4);
    \path [-,ultra thin] (1) edge node[left] {} (5);
    \path [-,ultra thin] (1) edge node[left] {} (6);
    \path [-,ultra thin] (1) edge node[left] {} (7);
    \path [-,ultra thin] (1) edge node[left] {} (8);
    \path [-,ultra thin] (1) edge node[left] {} (9);
    \path [-,ultra thin] (1) edge node[left] {} (10);
    \path [-,ultra thin] (2) edge node[left] {} (11);
    \path [-,ultra thin,dashed] (3) edge node[left] {} (11);
    \path [-,ultra thin] (3) edge node[left] {} (4);
    \path [ultra thin] (4) edge  node[left] {} (5);
    \path [-,ultra thin] (5) edge node[left] {} (6);
    \path [-,ultra thin] (6) edge node[left] {} (7);
    \path [-,ultra thin] (7) edge node[left] {} (8);
    \path [-,ultra thin] (8) edge node[left] {} (9);
    \path [-,ultra thin] (9) edge node[left] {} (10);
    \path [-,ultra thin] (10) edge node[left] {} (2);
    \path [-,ultra thin] (4) edge node[left] {} (14);
    \path [-,ultra thin] (4) edge node[left] {} (15);

\end{scope}

\begin{scope}[xshift=510]

  \node[circle, draw=black, fill=black, minimum size=0.1mm,scale=0.5](1) at (0,0){};
\node[below of=1,yshift=28,xshift= -2.5,label=:\small{$v$}]{};
    \path (1) ++(0:2) node[circle,draw=black, fill=black,scale=0.5, label=below:\small{$u_9$}](2) {};
    \path (1) ++(40:2) node[circle,draw=black, fill=black,scale=0.5, label=above:\small{$u_8$}](3){}; 

\node[below of=2,yshift=6,xshift=-12,label=:\tiny{$f_1$}]{};

\node[below of=3,yshift=2,xshift= -6,label=:\tiny{$f_2$}]{};

    \path (1) ++(80:2) node[circle,draw=black, fill=black,scale=0.5, label=above:\small{$u_7$}](4){}; 
    \path (1) ++(120:2) node[circle,draw=black, fill=black,scale=0.5, label=left:\small{$u_6$}](5){}; 
    \path (1) ++(160:2) node[circle,draw=black, fill=black,scale=0.5, label=left:\small{$u_5$}](6){}; 
    \path (1) ++(200:2) node[circle,draw=black, fill=black,scale=0.5, label=left:\small{$u_4$}](7){}; 
    \path (1) ++(240:2) node[circle,draw=black, fill=black,scale=0.5, label=left:\small{$u_3$}](8){};
    \path (1) ++(280:2) node[circle,draw=black, fill=black,scale=0.5, label=below:\small{$u_2$}](9){};
       \node[below of=9,label=below:$(e(ii))$]{};
    \path (1) ++(320:2) node[circle,draw=black, fill=black,scale=0.5, label=below:\small{$u_1$}](10) {};
    \path (1) ++(20:3.6) node[circle,draw=black, fill=black,scale=0.5, label = above :\small{$z$}](11) {};
    \path (2) ++(320:1.8) node[circle,draw=black, fill=black,scale=0.5, label = below :\small{$y$}](14){} ;

     \path (2) ++(20:1.8) node[circle,draw=black, fill=black,scale=0.5, label = right :\small{$x_1$}](15){} ;
        \path (2) ++(0:1.8) node[circle,draw=black, fill=black,scale=0.5, label = right :\small{$x_2$}](16){} ;
    \path (2) ++(-20:1.8) node[circle,draw=black, fill=black,scale=0.5, label = right :\small{$x_3$}](17){} ;
       

   \path [-,ultra thin] (1) edge node[left] {} (2);
    \path [-,ultra thin] (1) edge node[left] {} (3);


    \path [-,ultra thin] (1) edge node[left] {} (4);
    \path [-,ultra thin] (1) edge node[left] {} (5);
    \path [-,ultra thin] (1) edge node[left] {} (6);
    \path [-,ultra thin] (1) edge node[left] {} (7);
    \path [-,ultra thin] (1) edge node[left] {} (8);
    \path [-,ultra thin] (1) edge node[left] {} (9);
    \path [-,ultra thin] (1) edge node[left] {} (10);
     \path [-,ultra thin] (2) edge node[left] {} (11);
   \path [-,ultra thin,dashed] (3) edge node[left] {} (11);
    \path [-,ultra thin] (3) edge node[left] {} (4);
    \path [-,ultra thin] (4) edge node[left] {} (5);
    \path [-,ultra thin] (4) edge node[left] {} (5);
    \path [-,ultra thin] (5) edge node[left] {} (6);
    \path [-,ultra thin] (6) edge node[left] {} (7);
    \path [-,ultra thin] (7) edge node[left] {} (8);
    \path [-,ultra thin] (8) edge node[left] {} (9);
    \path [-,ultra thin] (9) edge node[left] {} (10);
    
    \path [-,ultra thin] (2) edge node[left] {} (14);
    \path [-,ultra thin, dashed] (10) edge node[left] {} (14);
    \path [-,ultra thin] (2) edge node[left] {} (15);
    \path [-,ultra thin] (2) edge node[left] {} (16);
   \path [-,ultra thin] (2) edge node[left] {} (17);

\end{scope}

\begin{scope}[yshift=-300]

    \node[circle, draw=black,scale=0.5, fill=black, minimum size=0.1mm](1) at (0,0){};
    \node[below of=1,yshift=28,xshift= -5.5,label=:\small{$v$}]{};
    \path (1) ++(0:2) node[circle,scale=0.5,draw=black, fill=black, minimum size=0.1mm, label=right:\small{$u_1$}](2) {};

\node[below of=2,yshift=4,xshift=-10,label=:\tiny{$f_1$}]{};

    \path (1) ++(45:2) node[circle,scale=0.5,draw=black, fill=black, label=above:\small{$u_8$}](3){}; 
    \path (1) ++(90:2) node[circle,scale=0.5,draw=black, fill=black](4){}; 
    \node[below of=4,yshift=28,xshift= 0,label=:\small{$u_7$}]{};
    \path (1) ++(135:2) node[circle,scale=0.5,draw=black, fill=black, label=above:\small{$u_6$}](5){}; 
    \path (1) ++(180:2) node[circle,scale=0.5,draw=black, fill=black, label=left:\small{$u_5$}](6){}; 
    \path (1) ++(225:2) node[circle,scale=0.5,draw=black, fill=black, label=below:\small{$u_4$}](7){}; 

     \path (4) ++(120:2) node[circle,scale=0.5,draw=black, fill=black, label=left:\small{$x_1$}](12){}; 
     \path (4) ++(60:2) node[circle,scale=0.5,draw=black, fill=black, label=right:\small{$x_2$}](13){};

\node[above of=7,yshift=-29,xshift=0,label=:\tiny{$f_2$}]{};
    
    \path (1) ++(270:2) node[circle,scale=0.5,draw=black, fill=black, label=below:$u_3$](8){};
\node[below of=8,label=below:$(f)$]{};

    \path (1) ++(315:2) node[circle,scale=0.5,draw=black, fill=black, label=below:\small{$u_2$}](9){};
    \path (2) ++(-45:2) node[circle,scale=0.5,draw=black, fill=black, label= right:\small{$z$}](10){};
    \path (6) ++(-135:2) node[circle,scale=0.5,draw=black, fill=black, label=left:\small{$y$}](11){};

    \path [-,ultra thin] (1) edge node[left] {} (2);
    \path [-,ultra thin] (1) edge node[left] {} (3);
    \path [-,ultra thin] (1) edge node[left] {} (4);
    \path [-,ultra thin] (1) edge node[left] {} (5);
    \path [-,ultra thin] (1) edge node[left] {} (6);
    \path [-,ultra thin] (1) edge node[left] {} (7);
    \path [-,ultra thin] (1) edge node[left] {} (8);
     \path [-,ultra thin] (1) edge node[left] {} (9);

     \path [-,ultra thin] (9) edge node[left] {} (10);
     \path [-,ultra thin] (7) edge node[left] {} (11);
     \path [-,ultra thin, dashed] (6) edge node[left] {} (11);
     \path [-,ultra thin] (2) edge node[left] {} (3);
     \path [-,ultra thin] (3) edge node[left] {} (4);
     \path [-,ultra thin] (4) edge node[left] {} (5);
    \path [-,ultra thin] (5) edge node[left] {} (6);
     \path [-,ultra thin] (7) edge node[left] {} (8);
     \path [-,ultra thin] (8) edge node[left] {} (9);
     \path [-,ultra thin, dashed] (2) edge node[left] {} (10);
     \path [-,ultra thin] (4) edge node[left] {} (12);
     \path [-,ultra thin] (4) edge node[left] {} (13);

\end{scope}

\begin{scope}[xshift=270,yshift=-300]

    \node[circle,scale=0.5, draw=black, fill=black, minimum size=0.1mm](1) at (0,0){};
      \node[below of=1,yshift=28,xshift= -5.5,label=:\small{$v$}]{};
    \path (1) ++(0:2) node[circle,scale=0.5,draw=black, fill=black, minimum size=0.1mm, label=right:\small{$u_1$}](2) {};

\node[above of=2,yshift=-28,label=below:\tiny{$f_1$}]{};

    \path (1) ++(45:2) node[circle,scale=0.5,draw=black, fill=black, label=right:\small{$u_8$}](3){}; 
    \path (1) ++(90:2) node[circle,scale=0.5,draw=black, fill=black, label=above:\small{$u_7$}](4){};

\node[above of =4,yshift=-25,xshift=10,label=below:\tiny{$f_3$}]{};

    \path (1) ++(135:2) node[circle,scale=0.5,draw=black, fill=black, label=above:\small{$u_6$}](5){}; 
    \path (1) ++(180:2) node[circle,scale=0.5,draw=black, fill=black, label=left:\small{$u_5$}](6){}; 

\node[above of=6,yshift=-28,label=below:\tiny{$f_2$}]{};

    \path (1) ++(225:2) node[circle,scale=0.5,draw=black, fill=black, label=below:\small{$u_4$}](7){}; 
    \path (1) ++(270:2) node[circle,scale=0.5,draw=black, fill=black, label=below:\small{$u_3$}](8){};

\node[below of=8,label=below:$(g)$]{};

    \path (1) ++(315:2) node[circle,scale=0.5,draw=black, fill=black, label=below:\small{$u_2$}](9){};
    \path (2) ++(-45:2) node[circle,scale=0.5,draw=black, fill=black, label=right:\small{$x$}](10){};
    \path (6) ++(-135:2) node[circle,scale=0.5,draw=black, fill=black,  label=left:\small{$y$}](11){};
     \path (3) ++(90:2) node[circle,scale=0.5,draw=black, fill=black,  label=above:\small{$z$}](12){};

    \path [-,ultra thin] (1) edge node[left] {} (2);
    \path [-,ultra thin] (1) edge node[left] {} (3);
    \path [-,ultra thin] (1) edge node[left] {} (4);
    \path [-,ultra thin] (1) edge node[left] {} (5);
    \path [-,ultra thin] (1) edge node[left] {} (6);
    \path [-,ultra thin] (1) edge node[left] {} (7);
    \path [-,ultra thin] (1) edge node[left] {} (8);
     \path [-,ultra thin] (1) edge node[left] {} (9);

     \path [-,ultra thin] (9) edge node[left] {} (10);
     \path [-,ultra thin] (7) edge node[left] {} (11);
     \path [-,ultra thin,dashed] (6) edge node[left] {} (11);
     \path [-,ultra thin] (2) edge node[left] {} (3);
    
     \path [-,ultra thin] (4) edge node[left] {} (5);
    \path [-,ultra thin] (5) edge node[left] {} (6);
     \path [-,ultra thin] (7) edge node[left] {} (8);
     \path [-,ultra thin] (8) edge node[left] {} (9);
     \path [-,ultra thin,dashed] (2) edge node[left] {} (10);
     \path [-,ultra thin,dashed] (3) edge node[left] {} (12);
     \path [-,ultra thin] (4) edge node[left] {} (12);

\end{scope}

\begin{scope}[xshift=510,yshift=-300]

    \node[circle,scale=0.5, draw=black, fill=black, minimum size=0.1mm, label=above:\small{$v$}](1) at (0,0){};
    \path (1) ++(0:2) node[circle,scale=0.5,draw=black, fill=black, minimum size=0.1mm, label=right:\small{$u_1$}](2) {};
    \path (1) ++(36:2) node[circle,scale=0.5,draw=black, fill=black, label=above:\small{$u_{10}$}](3){}; 
    \path (1) ++(72:2) node[circle,scale=0.5,draw=black, fill=black, label=above:\small{$u_9$}](4){}; 
    \path (1) ++(108:2) node[circle,scale=0.5,draw=black, fill=black, label=above:\small{$u_8$}](5){}; 
    \path (1) ++(144:2) node[circle,scale=0.5,draw=black, fill=black, label=left:\small{$u_7$}](6){}; 
    \path (1) ++(180:2) node[circle,scale=0.5,draw=black, fill=black, label=left:\small{$u_6$}](7){}; 
    \path (1) ++(216:2) node[circle,scale=0.5,draw=black, fill=black, label=left:\small{$u_5$}](8){};
    \path (1) ++(252:2) node[circle,scale=0.5,draw=black, fill=black, label=below:\small{$u_4$}](9){};

\node[below of=9,label=below:$(h)$,xshift=10]{};
    
    \path (1) ++(288:2) node[circle,scale=0.5,draw=black, fill=black, label=below:\small{$u_3$}](10){};
    \path (1) ++(324:2) node[circle,scale=0.5,draw=black, fill=black, label=right:\small{$u_2$}](11){};

     \path (2) ++(36:2) node[circle,scale=0.5,draw=black, fill=black, label=right:\small{$x$}](12){};

    \path [-,ultra thin] (1) edge node[left] {} (2);
    \path [-,ultra thin] (1) edge node[left] {} (3);
    \path [-,ultra thin] (1) edge node[left] {} (4);
    \path [-,ultra thin] (1) edge node[left] {} (5);
    \path [-,ultra thin] (1) edge node[left] {} (6);
    \path [-,ultra thin] (1) edge node[left] {} (7);
    \path [-,ultra thin] (1) edge node[left] {} (8);
     \path [-,ultra thin] (1) edge node[left] {} (9);
     \path [-,ultra thin] (1) edge node[left] {} (10);
     \path [-,ultra thin] (1) edge node[left] {} (11);

    \path [-,ultra thin] (3) edge node[left] {} (4);
     \path [-,ultra thin] (4) edge node[left] {} (5);
    \path [-,ultra thin] (5) edge node[left] {} (6);
    \path [-,ultra thin] (6) edge node[left] {} (7);
     \path [-,ultra thin] (7) edge node[left] {} (8);
     \path [-,ultra thin] (8) edge node[left] {} (9);
     \path [-,ultra thin] (9) edge node[left] {} (10);
     \path [-,ultra thin] (10) edge node[left] {} (11);
     \path [-,ultra thin] (2) edge node[left] {} (11);

     \path [-,ultra thin] (2) edge node[left] {} (12);
     \path [-,ultra thin,dashed] (3) edge node[left] {} (12);

\end{scope}

   \end{tikzpicture}
   \caption{Some configurations in a  planar graph with minimum degree three (dashed links are paths of length at least one)}
   \label{nineconfig}
   \end{figure}
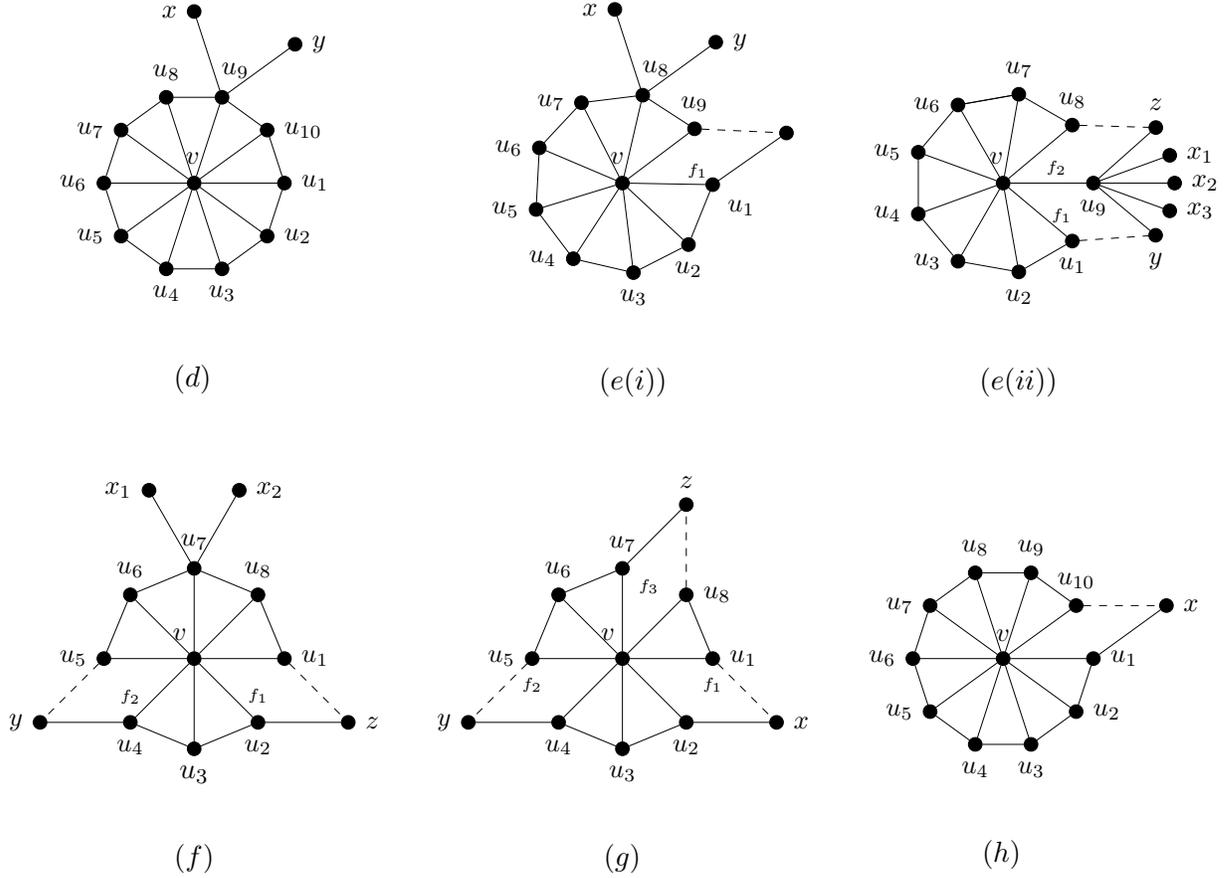
\end{center}

The motivating results for this paper include Hartnell and Rall's \cite{hartnell1994bounds} upper bound for the bondage number with respect to the degree sum of adjacent vertices. In 2000, Kang and Yuan \cite{kang2000bondage} proved that $b(G) \leq \min\{8, \Delta(G) + 2\}$ and $b(G) \leq 3$ for a graph $G$ with forbidden minor $K_4$. They also determined $b(G) \leq 7$ for a connected planar graph $G$ without vertices of degree 5. Fischermann et al. \cite{fischermann2003remarks} established that $b(G) \leq \Delta(G) + 1$ for connected planar graphs with $g(G) \geq 4$ and $\Delta(G) \geq 5$, and for all non-3-regular graphs with $g(G) \geq 5$.

Regarding the independent bondage number, there are fewer results compared to the bondage number. In 2018, Priddy, Wang, and Wei \cite{priddy2019independent} determined the independent bondage number for several graph classes and provided upper bounds using degree sums and maximum degree, as summarized in the following theorems:

\begin{theorem}\cite{priddy2019independent}  
\label{PriddyWei}
If $G$ is a non-empty graph, then
\[b_i(G) \leq \min\{d(u) + d(v) - |N(u) \cap N(v)| - 1: uv \in E(G)\}.\]
\end{theorem}

\begin{theorem}\cite{priddy2019independent}
If $G$ is a connected planar graph, then 
\[b_i(G) \leq \Delta(G) + 2.\]
\end{theorem}

In 2021, Pham and Wei \cite{phamindependent} determined a constant upper bound for the independent bondage number of a planar graph with $\delta(G) \geq 3$.

\begin{theorem}\cite{phamindependent}
\label{phamindependen}
Let $G$ be a planar graph with $\delta(G) \geq 3$, then
\[b_i(G) \leq 9.\]
\end{theorem}

In the same paper, they constructed a class of planar graphs with $b_i(G) = 6$ and $\delta(G) \geq 3$. We utilize the discharging method to improve Pham and Wei's result, and we prove the Theorem \ref{maintheorem} using the discharging rule, and then use Theorem \ref{maintheorem} to prove the Theorem \ref{IBNconfigurationproof}.

In the discharging method, we employ three distinct approaches for assigning charges, as detailed in Proposition \ref{dischargingequations}. These methods are derived from Euler's Formula, considering the number of vertices, edges, and faces of a planar graph (refer to \cite{cranston2017introduction} for more details about discharging method).

To establish configurations in a planar graph with a minimum degree of at least 3, as outlined in Theorem \ref{maintheorem}, we rely on the vertex charging technique.

\begin{proposition} \cite{cranston2017introduction}
\label{dischargingequations}
Let $V(G)$ and $F(G)$ be the set of vertices and faces respectively of a planar graph $G$. Denote $l(f)$ to be the length of a face $f$. Then the following equalities hold for $G$.
\begin{equation*}
\begin{aligned}
   &\sum_{\substack{v\in V(G)}}(d(v)-6) + \sum_{\substack{f\in F(G)}}(2l(f)-6) = -12  && \text{vertex charging}\\
   &\sum_{\substack{v\in V(G)}}(2d(v)-6) + \sum_{\substack{f\in F(G)}}(l(f)-6) = -12  && \text{face charging}\\
   &\sum_{\substack{v\in V(G)}}(d(v)-4) + \sum_{\substack{f\in F(G)}}(l(f)-4) = -8  && \text{balanced charging}
\end{aligned}
\end{equation*}
\end{proposition}

This paper is organized as follows. In Section 2, we implement discharging rules to prove Theorem \ref{maintheorem}, establish various facts and lemmas essential for the proof and Prove the Theorem \ref{maintheorem}. Section 3 is dedicated to proving preliminary lemmas and presenting additional facts crucial for the proof of Theorem \ref{IBNconfigurationproof}. Finally, in Section 4, we provide the proof for Theorem \ref{IBNconfigurationproof}.

\section{The Discharging Rules for Theorem \ref{maintheorem}}

We use the vertex-discharging method (the vertex charging equation in Proposition \ref{dischargingequations}) to establish Theorem~\ref{maintheorem}.

Let $G$ be a connected planar graph with $\delta(G) \geq 3$. Suppose that $G$ does not contain a configuration as described in the statement of Theorem~\ref{maintheorem}.
Initially assign a charge $d(v)-6$ to each vertex and a charge  $2d(f)-6$ to each face and then reassign vertex and face charges according to the following rules. In this section, we utilize notations introduced in Section \ref{introductiondefs}.

Additionally, we use the following notations. $f_{m^+}(v)$ represents an $m^+$-face incident with vertex $v$, and $f_{m^+}(u,v)$ represents an $m^+$-face incident with the edge $(u,v)$. $E_{m,n}(u,v)$ and $E_{m^+,n^+}(u,v)$ denote an edge $(u,v)$ incident with an $m$-face and $n$-face, and an $m^+$-face and $n^+$-face, respectively. $|E_{m,n}(u,v)|$ and $|E_{m^+,n^+}(u,v)|$ denote the number of edges $(u,v)$ incident with an $m$-face and $n$-face, and an $m^+$-face and $n^+$-face, respectively.
 
\begin{itemize}
    \item[(R1)] For a vertex $v$ with $d(v) = 3$, if there exists a $u \in N_{9^+}(v)$ and $E_{3,3}(u,v)$ (i.e., the edge $(u,v)$ is incident with two $3$-faces), then $v$ receives a charge of $1$ from $u$.
    
    \item[(R2)] Let $d(v)=3$. If $|F(v)|=1$, i.e., there is only one face,$f$, incident with $v$, then $v$ takes charge $3$ from the incident face $f$. If $|F(v)|\geq 2$, $v$ takes charge $1$, and $2$ from each incident $f_4(v)$ and $f_{5^+}(v)$, if any. 
    
    \item[(R3)] Let $d(v) = 3$ with $u \in N_{8^+}(v)$ and $E_{3,4^+}(u,v)$. If $d(u) = 8$ and $d(u) \neq 9$, or $d(u) = 9$ and $u$ does not fall into cases $(a)$ or $(b)$ as described below, then $v$ receives a charge of $\frac{1}{2}$ from $u$. Otherwise, if $d(u) = 9$ and $u$ falls into cases $(a)$ or $(b)$, then $v$ receives a charge of $\frac{1}{2}$ from $f_{4^+}(u,v)$, the $4^+$-face incident with the $(u,v)$-edge.

\textbf{Cases:}    
    \begin{itemize}
        \item[(a)]  $|F_{4^+}(u)|=1$, $|F_3(u)|=8$, and
        \begin{itemize}
\item[($\alpha$)] The vertex $w\in N(u)\setminus {v}$ has $E_{3,4^+}(u,w)$, indicating that $w$ is the other neighbor of $u$ on the $4^+$-face, and $d(w)\geq 8$, and
\item[($\beta$)] $|E_{3,3}(u,w)|\geq 3$, where $w\in N_3(u)$, meaning that there are at least three edges $(u,w)$ incident on two 3-faces, where $w\in N(u)$ and $d(w)=3$ (following the structure depicted in Figure \ref{nineconfig} $(e)(i)$, where we designate vertex $u_9$ as the $3$-vertex $v$, the center vertex $v$ as the degree $9$ neighbor $u$ of the $3$-vertex, the vertex $u_1$ as the $8^+$-vertex $w$, and specify $d(u_i) = 3$ for $i = 3,5,7$).
        \end{itemize}
        \item[(b)] $|F_{4^+}(u)|=2$ and $|F_3(u)|=7$, and 
         \begin{itemize}
             \item[($\alpha$)]
       $d(w)\geq 7$, where $w\in N(u)$ and $E_{4^+,4^+}(u,w)$, indicating the neighbor of $u$ on the two $4^+$-faces, and
             \item[($\beta$)] $|E_{3,3}(u,w)|\geq 3$, where $w\in N_3(u)$ (following the structure depicted in Figure \ref{nineconfig} $(e)(ii)$, where we designate vertex $u_8$ as the $3$-vertex $v$, the center vertex $v$ as the degree $9$ neighbor $u$ of the $3$-vertex, the vertex $u_9$ as the $7^+$-vertex $w$, and specify $d(u_i) = 3$ for $i = 2,4,6$).
         \end{itemize}
          \end{itemize} 
  
    \item[(R4)] Let $v$ be a  4-vertex. If $|F_{4^+}(v)|\leq 1$,
     then $v$ takes charge 1, or 2 from the $f_{4^+}(v)$ (if any) depending on whether it is a 4-face or 5-face respectively. Also, if there is $E_{3,3}(v,u)$ for $u\in N(v)$ with $d(u)\geq 8$, then $v$ takes charge $\frac{1}{2}$  from $u$.
    \item[(R5)]  Let $v$ be a 4-vertex. If $|F_{4^+}(v)|\geq 2$, then 
   take charge $\frac{2}{|F_{4^+}(v)|}$ from each  incident $4^+$-face.
   
     \item[(R6)] Let $v$ be a 5-vertex with $|F_3(v)|=5$, then $v$ takes charge  $\frac{1}{5}$ from each $N_{7^+}(v)$
    \item[(R7)] Let $v$ be a 5-vertex with $|F_{4^+}(v)|\geq 1$, then $v$  takes charge  $\frac{1}{|F_{4^+}(v)|}$ from each $F_{4^+}(v)$.

\end{itemize}
We next show, after redistributing the charges according to the rules (R1) to (R7),  that each vertex and face has a non-negative charge by establishing the following facts. 
Let $v$ be a vertex of $G$ throughout the remainder of the proof of Theorem~\ref{maintheorem}.
Note that $|F(v)| \leq d(v)$ with strict inequality only if $v$ is a cut-vertex. Without loss of generality, let $u_i \in N(v)$ for $i \in \{1, 2, \dots, d(v)\}$ along the clockwise direction around the neighborhood of $v$.

\begin{fact}
\label{fact1}
If  $u_i\in N_{5^-}(v)$, then $d(u_j)\geq 6$ for $u_j\in N(v)\cap N(u_i)$.
\end{fact}

\begin{proof}
  This fact follows since $G$ avoids configurations (b) and (c).    
\end{proof}

\begin{fact}
    \label{fact2}
    If $v$ is a  $7^-$-vertex, then $v$ has a non-negative charge. 
\end{fact}

\begin{proof} 
    Assume by contradiction that $v$ has a negative charge. \\
    
\textbf{Case 1:} $d(v)=7$. 
Then $v$ starts with charge $d(v)-6=1$ so that $v$ loses charge greater than 1. 
The vertex $v$ only loses charge $\frac{1}{5}$ by
(R6) for each $w \in N_5(v)$ with 
$F(w)=F_3(w)$ (note that $7^-$-vertices does not lose charges to $N_3(v)$ and $N_4(v)$). Thus there are at least six such $w$ vertices. But no two of these vertices are adjacent by Fact 1. 
Hence $v\cup N(v)$ contains a $12$-fan with hub $v$ contradicting that $d(v)=7$. \\

\textbf{Case 2:}  $d(v)=6$. Then $v$ begins and ends with non-negative charge $0$ as the rules (R1)-(R7) do not apply to $v$; a contradiction.\\

\textbf{Case 3:} $d(v)=5$. Then $v$ begins with charge $-1$ and hence receives a total charge less than $1$. Thus $|F_{4^+}(v)|=0$ by (R7). Then $F(v)=F_3(v)$ so that  $v \cup N(v)$ is a 5-wheel. Each rim vertex of the wheel is a $7^+$-vertex as $G$  avoids configuration (a). Thus $v$ receives charge $\frac{1}{5}$ from each of its five neighbors  by  (R6); a contradiction.\\

\textbf{Case 4:} $d(v)=4$. Then $v$ starts with charge $-2$ and receives a total charge less than $2$. Thus $(R5)$ implies that each face meeting $v$ is a 3-face except for possibly one $4^+$-face. Now, $(R4)$ implies $|F_{5^+}(v)|=0$. Thus, $|F_4(v)|\leq 1$. 
If $v$ is a cut-vertex, then the minimum degree condition implies that  $|F_{5^+}(v)|\geq 1$; a contradiction. Thus $v$ is not a cut-vertex and $|F(v)|=4=d(v)$. Thus, Either $|F_3(v)|=3$ and $|F_4(v)|=1$, or $|F_3(v)|=4$. 

\textbf{Case 4.1} $|F_3(v)|=3$ and $|F_4(v)|=1$. Then, there is a $4$-fan $\{v,u_1,u_2,u_3,u_4\}$ consisting of $v\cup N(v)$.  Thus, $E_{3,3}(v,u_i)$ with $u_i\in N_{8^+}(v)$ for $i\in \{2,3\}$ as $G$ avoids the configuration $(a)$. Thus, $v$ receive charge $\frac{1}{2}$ from each of $u_i$ for $i\in \{2,3\}$ by $(R4)$, $1$ from $4$-face incidents with $v$; a contradiction. \\

\textbf{Case 4.2} $|F_3(v)|=4$, $(u_1,u_4) \in E(G)$. Then, the set  $\{v,u_1,u_2,u_3,u_4\}$ is an 4-{\it wheel}. Thus, $E_{3,3}(v,u_i)$ with $u_i\in N_{8^+}(v)$ for $i\in \{12,3,4\}$ as $G$ avoids the configuration $(a)$. Hence $v$ receives charge $\frac{1}{2}$ from each $u_i$ for $i\in \{12,3,4\}$ by $(R4)$; a contradiction.
 
\vspace{0.3cm}

\textbf{Case 5:} $d(v)=3$. Then, $v$ starts with a charge of $-3$ and receives a charge less than $3$. \\

\textbf{Case 5.1:} Suppose that $v$ is a cut-vertex. Thus,  $G-\{v\}$ has two or three components. i.e. $|F(v)|\leq2$.\\

\textbf{Case 5.1.1:} First consider $G-\{v\}$ has two components $G_1$, and $G_2$. i.e. $|F(v)=2|$. Note that in this case, $v$ incidents with a cut edge $e$. Since $G$ is simple and planar with $\delta(G)\geq 3$, and $G$ avoids configuration $(c)$, $G_1, G_2$ are simple and planar with  $\delta(G_1),\delta(G_2)\geq 3$. Thus, $G_1, G_2$ contains a cycle. So, $l(f_1),l(f_2)\geq 3$, where $f_1,f_2$ are the outer faces of $G_1$ and $G_2$ respectively. Thus in $G$, we have a $f_{8^+}(v)$ that consists of the cut edge $e$, and the two $f_{3^+}(v)$ faces ( note that the cut edge count twice to the length of the outer face). Then $v$ receives charge 2 from the $f_{8^+}(v)$-face by $(R2)$. If $f_{3^+}(v)$ is a $3$-face, then $u_i\in N_{8^+}(v)$ for $i\in\{ 1,2\}$ where $u_i\in f_{3,8^+}(v,u_i)$ as $G$ avoids the configuration $(b)$ ( if there is an edge $(3,u)$ with $E_{3,4^+}(3,u)$), then $d(u)\geq 8$). Thus $v$ receive charge $\frac{1}{2}$ from each $u_i$ for $i\in \{ 1,2\}$ by $(R3)$; a contradiction. If $f_{3^+}(v)$ is a $4^+$-face, then $v$ receive charge $1$ from the $4^+$-face by $(R2)$; a contradiction.\\

\textbf{Case 5.1.2:} Now consider that case when $G-\{v\}$ has three components, then $|F(v)|=1$. Thus $v$ receives charge $3$ by $(R2)$; a contradiction.\\

\textbf{Case 5.2:} Now suppose $v$ is not a cut vertex. Thus, $|F(v)|=3$. If $|F_{4^+}(v)|=3$, or $|F_{5^+}(v)|=1$ and $|F_{4^+}(v)|\geq 1$, then $v$ receives a total of charge $3$ by $(R2)$; a contradiction. Thus $|F_{4}(v)|\leq 2$ and $|F_3(v)|\geq 1$, and $|F_{5^+}(v)|=1$ and $|F_3(v)=2|$, or $|F_{5^+}(v)|=1$ and $|F_3(v)|=2$.

\textbf{Case 5.2.1:} Consider $|F_4(v)|=2$ and $|F_3(v)|=1$. Then, there are two 
 such that $d(u_i)\geq 8$ where  $E_{3,4}(v,u_i)$ for $i\in \{1,2\}$. Thus, $v$ receives a total of charge $1$ from $u_i$ or from the $4^+$-face incident with the $(v,u_i)$-edge for $i\in \{1,2\}$ by $(R3)$, and a total charge of $ 2 $ from the two $F_{4}(v)$ by $(R2)$; a contradiction. 
 
\textbf{Case 5.2.2:} Consider $|F_4(v)|=1$, and $|F_3(v)|=2$. Then $d(u)\geq 9$ for $u\in N(v)$ with $E_{3,3}(v,u)$ since $G$ avoids the configuration $(a)$, and $d(u_i)\geq 8$ with $E_{3,4}(v,u_i)$ for $i\in \{1,2\}$. Thus, $v$ receive charge $1$ from $u$ by $(R1)$, and a total of charge $1$ from $u_i$ or from the $4^+$-face incident with $(v,u_i)$ for $i\in \{1,2\}$ by $(R3)$, and $1$ from $f_4(v)$ by $(R2)$, which total upto charge $3$; a contradiction. 

\textbf{Case 5.2.3:} Consider  $|F_4(v)|=0$ and $|F_3(v)|=3$. Then, there is a 3-fan $\{v,u_1,u_2,u_3\}$ consisting of $v\cup N(v)$. This set is a 3-wheel. The graph $G$ avoids configuration $(a)$ so that $d(u_i)\geq 9$ for $i \in \{1,2,3\}$. Thus, $v$ receives charge $1$ from each $u_i$ for $i \in \{1,2,3\}$ by $(R1)$; a contradiction. 

\textbf{Case 5.2.4:} Consider $|F_{5^+}(v)|=1$ and $|F_3(v)|=2$. since $G$ avoids the configuration $(a)$, $d(u)\geq 9$ for $E_{3,3}(v,u)$ providing that $v$ receive charge $1$ from $u$ by $(R1)$ and charge $2$ from $f_{5^+}(v)$ by $(R2)$; a contradiction.

This completes the proof of Fact 2.
\end{proof}

\begin{fact}
\label{fact3}
     Suppose  $d(v)\geq 8$ and  $\{v, u_1,u_2,\dots,u_k\}$  is a $k$-fan contained in $v\cup N(v)$ such that $E_{3,4^+}(u_i,v)$ for $i\in \{1,2,\dots,k\}$, and $E_{3,3}(u_i,v)$ for $1<i<k$ for  $k\geq 3$ and $(u_1,v),(u_k,v)$ are in two consecutive $F_4(v)$s for $k\leq 2$ (see Figure \ref{fact3picture} $(u_3,v)$-edge for $k=1$, and $(u_1,v)$ and $(u_2,v)$-edges for $K=2$), then $v$ can lose at most $\frac{k-1}{2}$ charge to $\{u_1,u_2,\dots,u_k\}$. In particular, $v$ lose at most  $\frac{k}{4}$ and $\frac{k+1}{4}$ charge to  $\{u_1,u_2,\dots,u_k\}$ when $k$ is even and $k$ is odd, respectively when $d(v)=8$, and at most $\frac{k-1}{2}$ charge to $\{u_1,u_2,\dots,u_k\}$ when $d(v)\geq 9$.
\end{fact}

   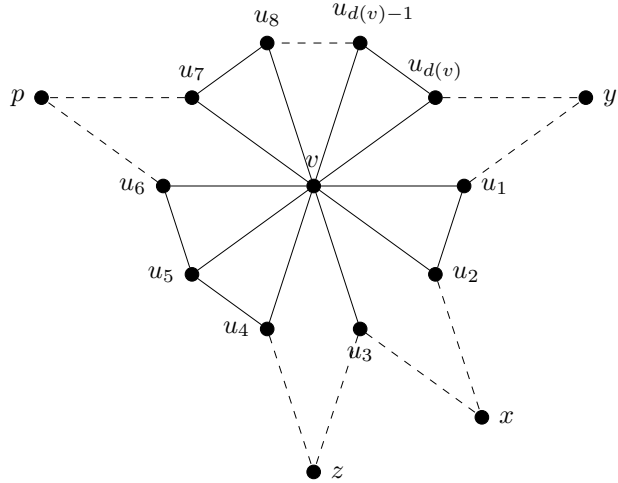
\begin{figure}[h]
   \label{fact3picture}
   \begin{center}
\begin{tikzpicture}[scale=1]

    \node[circle,scale=0.5, draw=black, fill=black, minimum size=0.1mm](1) at (0,0){};
\node[below of=1,yshift=28,xshift= -.5,label=:\small{$v$}]{};
    \path (1) ++(0:2) node[circle,scale=0.5,draw=black, fill=black, minimum size=0.1mm, label=right:\small{$u_1$}](2) {};
    \path (1) ++(36:2) node[circle,scale=0.5,draw=black, fill=black, label=above:\small{$u_{d(v)}$}](3){}; 
    \path (1) ++(72:2) node[circle,scale=0.5,draw=black, fill=black](4){}; 
     \node[below of=4,yshift=28,xshift= 5,label= above:\small{$u_{d(v)-1}$}]{};
    \path (1) ++(108:2) node[circle,scale=0.5,draw=black, fill=black, label=above:\small{$u_8$}](5){}; 
   
    \path (1) ++(144:2) node[circle,scale=0.5,draw=black, fill=black, label=above:\small{$u_7$}](6){}; 
    \path (1) ++(180:2) node[circle,scale=0.5,draw=black, fill=black, label=left:\small{$u_6$}](7){}; 
    \path (1) ++(216:2) node[circle,scale=0.5,draw=black, fill=black, label=left:\small{$u_5$}](8){};
   
    \path (1) ++(252:2) node[circle,scale=0.5,draw=black, fill=black, label=left:\small{$u_4$}](9){};

    \path (1) ++(288:2) node[circle,scale=0.5,draw=black, fill=black, label=below:\small{$u_3$}](10){};
    
    \path (1) ++(324:2) node[circle,scale=0.5,draw=black, fill=black, label=right:\small{$u_2$}](11){};
    
     \path (11) ++(-72:2) node[circle,scale=0.5,draw=black, fill=black, label=right:\small{$x$}](12){};

       \path (2) ++(36:2) node[circle,scale=0.5,draw=black, fill=black, label=right:\small{$y$}](13){};

      \path (9) ++(-72:2) node[circle,scale=0.5,draw=black, fill=black, label=right:\small{$z$}](17){};
        \path (7) ++(144:2) node[circle,scale=0.5,draw=black, fill=black, label=left:\small{$p$}](18){};
    
    \path [-,ultra thin] (1) edge node[left] {} (2);
    \path [-,ultra thin] (1) edge node[left] {} (3);
    \path [-,ultra thin] (1) edge node[left] {} (4);
    \path [-,ultra thin] (1) edge node[left] {} (5);
    \path [-,ultra thin] (1) edge node[left] {} (6);
    \path [-,ultra thin] (1) edge node[left] {} (7);
    \path [-,ultra thin] (1) edge node[left] {} (8);
     \path [-,ultra thin] (1) edge node[left] {} (9);
     \path [-,ultra thin] (1) edge node[left] {} (10);
     \path [-,ultra thin] (1) edge node[left] {} (11);  
    \path [-,ultra thin] (3) edge node[left] {} (4);
     \path [-,dashed, ultra thin] (4) edge node[left] {} (5);
    \path [-,ultra thin] (5) edge node[left] {} (6);
    
     \path [-,ultra thin] (7) edge node[left] {} (8);
     \path [-,ultra thin] (8) edge node[left] {} (9);  
    
     \path [-,ultra thin] (2) edge node[left] {} (11);

\path [-, dashed, ultra thin] (11) edge node[left] {} (12);
\path [-,dashed, ultra thin] (10) edge node[left] {} (12);
  \path [-,dashed, ultra thin] (2) edge node[left] {} (13);
  \path [-,dashed, ultra thin] (3) edge node[left] {} (13);
  \path [-,dashed,ultra thin] (9) edge node[left] {} (17);
     \path [-,dashed, ultra thin] (10) edge node[left] {} (17);
       \path [-,dashed, ultra thin] (6) edge node[left] {} (18);
       \path [-,dashed, ultra thin] (7) edge node[left] {} (18);
     
   \end{tikzpicture}
   \caption{$1,2,3$, and $k$-fans around $v\cup N(v)$ (dashed links are paths of length at least one)}
   \label{fact3}
   \end{center}
   \end{figure}

\begin{proof}
 Let $G$ be graph without having the configurations $(a)-(h)$, and  $d(v)\geq 8$ and  $\{v, u_1,u_2,\dots,u_k\}$  is a $k$-fan contained in $v\cup N(v)$ such that edges $(u_1,v)$ and $(u_k,v)$ are on $F_4(v)$s. Let $k=1$. Then we have $f_{4^+,4^+}(u_1,v)$. Suppose $d(u_1)\leq 5$. $v$ does not give charge to $u_1$ by $(R1)-(R7)$. Thus, the result is true for $K=1$. Suppose $k=2$. Without loss of generality assume $d(u_1)=3$, then $d(u_2)\geq 8$ since $G$ avoids the configuration $(b)$. Thus, $v$ lose at most $\frac{1}{2}=\frac{2-1}{2}=\frac{k-1}{2}$ when $k=2$. Now consider the case when $k\geq 3$. Then $(u_1,v),( u_k,v)$ on a $F_{4^+}(v)$ and a $F_3(v)$, and have $f_{3,3}(u_i,v)$. Since $G$ avoids the configuration $(a), (b)$ the $k$-fan have at most $\lfloor\frac{k+1}{2}\rfloor$ of $5^-$-vertices. Consider the case when $k$ is odd. Then, $V_1=\{u_1, u_3, u_5,\dots, u_{k-2},u_k\}$, or $V_2=\{u_2, u_4, u_6,\dots, u_{k-3}, u_{k-1}\}$ are the maximal set of $5^-$-vertices in $k-$fan. Note that since we avoid the configurations $(a),(b)$, and $(c)$, it is enough to consider maximal possible $N_3(v)$ and $N_4(v)$ on the $k$-fan when $d(v)\geq 9$, and $d(v)=8$ respectively. Thus, if $d(v)\geq 9$, $v$ lose at most $\frac{k-3}{2}+2*\frac{1}{2}=\frac{k-1}{2}$ , and at most $\frac{k-1}{2}$ for $V_1$ and $V_2$ respectively. Now, if $d(v)=8$, then $v$ lose at most $\frac{1}{2}*\frac{k-3}{2}=\frac{k-3}{4}$, or at most $\frac{1}{2}*\frac{k-1}{2}=\frac{k-1}{4}$, for $V_1$, and $V_2$ respectively (note that when $d(v)=8$, $v$ does not lose charge to $u_1$ where $f_{3,4^+}(u_1,v)$). Thus, when $k$ is odd with $d(v)=8$, $v$ lose at most $\frac{k-1}{4}$. Thus, the result is true when $k$ is odd. Now consider the case when $k$ is even. Then, $\{u_1, u_3, u_5,\dots, u_{k-3},u_{k-1}\}$, or $\{u_2, u_4, u_6, \dots, u_{k-2}, u_k\}$ are the maximal set of $5^-$-vertices in $k-$fan. Note that here we have $u_1$ or $u_k$ on $F_3(v)$ and $F_{4^+}(v)$ but not both. Thus, $v$ lose at most $\frac{1}{2}$ to these vertices by $(R3)$, and $v$ lose $1$ to each of the remaining $\frac{k-2}{2}$ vertices. Thus $v$ lose at most $\frac{1}{2}+\frac{k-2}{2}=\frac{k-1}{2}$ to the $k$-fan. Thus we have the result.   
\end{proof}

\begin{fact}
    \label{fact4}
     Every $v$ with $8\leq d(v)\leq 12$ with $|F_3(v)|=d(v)$ ends with a non-negative charge.
\end{fact}

\begin{proof} of Fact \ref{fact4}
 Let $G$ be graph without having the configurations $(a)-(h)$, and $v\in V(G)$ with $8\leq d(v)\leq 12$ with $|F_3(v)|=d(v)$. Suppose by contradiction that $v$ ends with a negative charge. i.e. $v$ lose more than $d(v)-6$. Thus, $v$ loses more than $2, 3, 4, 5,$ and $6$ charge from $8,9,10,11$, and $12$-vertices respectively. Now, consider the case when $d(v)=8$. There is a $8$-fan $\{v,u_1,u_2,\dots, u_{8}\}$ consisting of $v\cup N(v)$. By Fact 1, $|N_{5^-}(v)|\leq \lfloor \frac{d(v)}{2} \rfloor=2$. Thus, $|E_{3,3}(v,u)|\leq4$ for $u\in N(v)$ with $d(u)\leq 5$. $v$ lose at most charge $2$ by $(R4)$; contradiction. Suppose $d(v)=9$. By the same argument,  $|E_{3,3}(v,u)|\leq4$ for $u\in N(v)$ with $d(u)\leq 5$. Since $G$ avoids the configuration $(d)$,  $|N_{5^-}(v)|\leq 1$, $|N_3(v)|\leq 2$ and $|N_{5^-}(v)|\leq 1$, or $|N_{3}(v)|\leq 1$ and $|N_{5^-}(v)|\leq 2$. Thus, $v$ loses at most $3$; a contradiction. Consider $d(v)=10$.  $|E_{3,3}(v,u)|\leq5 $ for $u\in N(v)$ with $d(u)\leq 5$. Since $G$ avoids the configuration $(d)$,  $|N_{3}(v)|\leq 4$, $|N_{3}(v)|\leq 3$ and $|N_{5^-}(v)|\leq 1$, $|N_{3}(v)|\leq 3$ and $|N_{5^-}(v)|\leq 1$, or $|N_{3}(v)|\leq 2$ and $|N_{5^-}(v)|\leq 2$. Thus, $v$ loses at most charge 4; a contradiction. Now, consider $11 \leq d(v)\geq 12$.  $|E_{3,3}(v,u)|\leq \lfloor\frac{d(v)}{2}\rfloor$ for $u\in N(v)$ with $d(u)\leq 5$. Thus, $v$ lose at most $d(v)-6- \lfloor\frac{d(v)}{2}\rfloor\geq 0$ for $d(v)\geq 11$; a contradiction.    
\end{proof}

\begin{lemma} \label{lem1}
If $v$ be such that $d(v)\geq 8$, then $v$ loses at most $\frac{d(v)-|F_{4^+}(v)|}{2}$ charge to its neighbors.
\end{lemma}
\begin{proof}
First assume that $F(v)=F_3(v)$. i.e. the case when $|F_{4^+}(v)|=0$. Then, there are at most $\lfloor\frac{d(v)}{2}\rfloor$  of $(v,u)$ edges such that $d(u)\leq 5$ with $E_{3,3}(v,u)$ since we avoid the configuration $(b)$. For $d(v)=8$, $v$ lose charge to $N_4(v)$, and $N_5(v)$ by $(R4)$, and $(R5)$ respectively, and thus $v$ lose highest charge to $N_4(v)$. So, $v$ lose at most  $\frac{d(v)}{2}*\frac{1}{2}=\frac{d(v)}{4}<\frac{d(v)}{2}$ to $N(v)$ when $d(v)=8$. For $d(v)\geq 9$, $v$ lose highest charge to $N_3(v)$ by $(R1)$, and thus $v$ lose at most $\frac{d(v)}{2}$ when $d(v)\geq 9$. Thus, when $|F_{4^+}(v)|=0$, the result is true.\\
When $|F_{4^+}(v)|=1$, let $f_1=F_{4^+}(v)$, let $u_1,u_{d(v)}\in N(v)$ such that $(v,u_1),(v,u_{d(v)})$ incidents with $f_1$. Then, according to the Fact 3, $k=d(v)$, thus $v$ lose at most $\frac{d(v)-1}{2}$. So, when $|F_{4^+}(v)|=1$, the result is true. Now consider $|F_{4^+}(v)|=s$. Let $\{f_1,f_2,\dots, f_s\}$ be $F_{4^+}(v)$ faces, and the edges $(v,u_1)$ and $(v,u_{d(v)})$ incidents with $f_1$, $(v,u_{k_1})$ and $(v,u_{k_1+1})$ incidents with $f_2$, $(v,u_{k_2})$ and $(v,u_{k_2+1})$ incidents with $f_3$, $\dots$,  $(v,u_{k_{i-1}})$ and $(v,u_{k_{i-1}+1})$ incidents with $f_i$, $\dots$,$(v,u_{k_{s-2}})$ and $(v,u_{k_{s-2}+1})$ incidents with $f_{s-1}$ , and $(v,u_{k_{s-1}})$ and $(v,u_{k_{s-1}+1})$ incidents with $f_s$. By Fact 2, $v$ lose at most $\frac{k_1-1}{2}$ to the vertices $\{u_1,u_2,\dots, u_{k_{1}}\}$, $\frac{k_2-(k_1+1)}{2}$ to the vertices $\{u_{k_1},u_{{k_1}+1},\dots, u_{k_{2}}\}$, $\frac{k_3-(k_2+1)}{2}$ to the vertices $\{u_{k_2+1},u_{{k_2}+2},\dots, u_{k_{3}}\}$, $\dots,$ $\frac{k_{i}-(k_{i-1}+1)}{2}$ to the vertices $\{u_{k_{{i-1}+1}},u_{k_{{i-1}+2}},\dots,u_{{k_i}}\},\dots, $ $\frac{k_{s-1}-(k_{s-2}+1)}{2}\}$ to the vertices\\ $\{u_{k_{{s-2}+1}},u_{k_{{s-2}+2}},\dots,u_{{k_{s-1}}} \}$, and $\frac{d(v)-(k_{s-1}+1)}{2}$ to the vertices $\{u_{k_{{s-1}+1}},u_{k_{{s-1}+1}},\dots, u_{d(v)} \}$. Thus, $v$ lose at most, $\frac{k_1-1}{2}+\frac{k_2-(k_1+1)}{2}+\frac{k_3-(k_2+1)}{2}+\dots+\frac{k_{i}-(k_{i-1}+1)}{2}+\dots+\frac{k_{s-1}-(k_{s-2}+1)}{2}+\frac{d(u)-(k_{s-1}+1)}{2}=\frac{d(v)-s}{2}$. \\

\end{proof}

\begin{fact}
    \label{fact5}
    If $v\in V(G)$ such that $d(v)\geq 8$ with  $d(v) + |F_{4^+}(v)|\geq$ 12, then $v$ ends with a non-negative charge.

\end{fact}

\begin{proof}
 By Lemma \ref{lem1} the vertex $v$ has charge at least  $d(v)-6-(\frac{d(v)-|F_{4^+}(v)|}{2})=\frac{d(v)+|F_{4^+}(v)|-12}{2}\geq 0$. i.e. $d(v) + |F_{4^+}(v)|\geq$ 12.   
\end{proof}

\begin{fact}
    \label{fact6}
    Any $v\in V(G)$ such that $d(v)\geq 10$ ends with a non-negative charge.
\end{fact}

\begin{proof}
     If $d(v)\geq 12$, then by Fact 5, $d(v)+|F_{4^+}(v)|\geq 12$ for $|F_{4^+}(v)|\geq 0$, and thus have the result. Let $d(v)=11$. If $|F_{4^+}(v)|\geq 1$, then we have the result by the Fact 4. Thus, consider the case when $|F_{4^+}(v)|=0$. Then we have the result by Fact 4. Now consider $d(v)=10$. Similarly, if $|F_{4^+}(v)|\geq 2$, we have the result by Fact 5. Thus, we only have to deal with $|F_{4^+}(v)|\leq 1$. If $|F_{4^+}(v)|=1$, by Fact 3, $k=11$ and $v$ lose at most $\frac{11-1}{2}=5$. Note that $v$ starts with charge $11-6=5$, and thus it ends with a non-negative charge. Now consider the case when $|F_{4^+}(v)|=0$. Now, by Fact 4, $v$ ends with a non-negative charge. Thus we have the result.
\end{proof}

\begin{fact}
    \label{fact7}
    Any $v\in V(G)$ such that $d(v)= 9$ ends with a non-negative charge.
\end{fact}

\begin{proof}
  Let $v\in V(G)$ such that $d(v)= 9$. By Fact 5, $d(v)+|F_{4^+}(v)|\geq 12$ for $|F_{4^+}(v)|\geq 3$. Thus, we only have to deal with $|F_{4^+}(v)|\leq 2$. First, consider the case when $|F_{4^+}(v)|=0$. Then, by Fact 4, we have the result. Secondly, consider the case when $|F_{4^+}(v)|=1$. Note that $E_{3,3}(u_i,v)$ for $i\in \{2,3,\dots,8\}$. Thus if $d(u_i)\leq 5$, $u_i$ takes charge from $v$ according to $(R1), (R4)$, and $(R6)$. We have $E_{4^+,3}(u_i,v)$ for $i\in\{1,9\}$, thus $u_i$ take charge from $v$ only by $(R3)$. Thus, $v$ may lose charge $\frac{1}{2}$ to $\{u_1,u_9\}$ when $\{u_1,u_9\}\subseteq N_3(v)$. Assume by contradiction that $v$ has a negative charge. Thus $v$ loses at least $3\frac{1}{5}$ to $N_{5^-} (v)$. Since $G$ avoids the configuration $(b)$, $|N_{5^-}(v)|\leq 5$. If $|N_{5^-}(v)|\leq 3$, $v$ lose at most 3. Thus, $|N_{5^-}(v)|\geq 4$. Suppose $|N_{5^-}(v)|=4$. Now since $v$ loses at least  $3\frac{1}{5}$, $|N_3(v)|=3$ and another $|N_{5^-}(v)|=1$. Since $G$ avoids the configuration $(b)$, $N_{5^-}(v)$ are independent. Thus, this forces to have the configuration $(e)(i)(\alpha)$, or $(e)(i)(\delta)$; a contradiction. Now assume $|N_{5^-}(v)|=5$. Here $|E_{3,3}(u_i,v)|=3$ and $|E_{3,4^+}(u_i,v)|\geq 1$ where $u_i \in N_3(v)$, or $|E_{3,3}(u_i,v)|=2$ with $u_i \in N_3(v)$ and another $|E_{3,3}(u_i,v)|=1$ with $u_i \in N_{5^-}(v)$, and $|E_{3,4^+}(u_i,v)|=2$ with $u_i \in N_3(v)$. Each of these cases is included in $(e)(i)(\beta)$; a contradiction since  $G$ avoids this configuration. Thus, we have the result.
  
\end{proof}

\begin{fact}
    \label{fact8}
    Any $v\in V(G)$ such that $d(v)= 8$ ends with a non-negative charge.
\end{fact}

\begin{proof}
  Let $v\in V(G)$ such that $d(v)= 8$. By Fact 5, $d(v)+|F_{4^+}(v)|\geq 12$ for $|F_{4^+}(v)|\geq 4$. Thus, we only have to deal with $|F_{4^+}(v)|\leq 3$.  First, consider the case when $|F_{4^+}(v)|=0$. Then, by Fact 4, we have the result. Now, suppose $|F_{4^+}(v)|=1$, then $\{v,u_1,u_2,\dots, u_8\}$ is an $8$-fan contained in $v\cup N(v)$ having $E_{3,4^+}(u_1,v)$ and $E_{3,4^+}(u_8,v)$. Thus according to Fact $3$, $k=8$. So, $v$ lose at most $\frac{8}{4}=2$. Now, consider the case when $|F_{4^+}(v)|=2$. Then, $n\cup N(v)$ consists of $1$-fan and $7$-fan, $2$-fan and $6$-fan, $3$-fan and $5$-fan, or two $4$-fans. We can use the Fact $ 3$ for each of the above cases. Thus, those cases becomes $k=1$ and $k=7$, $k=2$ and $k=6$, $k=3$ and $k=5$, and $k=4$ and $k=4$. When $k=1$ and $k=7$, $v$ lose at most $\frac{7+1}{4}=2$ by Fact $3$.  When $k=2$ and $k=6$, $v$ lose at most $\frac{2}{4}+\frac{6}{4}=2$ by Fact $3$.  Note that $G$ avoids the configuration $(f)$. Thus, when $k=3$ and $k=5$, $|E_{3,4}(u,v)|\leq4$, and $|E_{3,3}(u,v)|=0$ or $|E_{3,4}(u,v)|\leq 3$, and $|E_{3,3}(u,v)|=1$. Thus, $v$ lose at most 2 according to $(R4)$ and $(R7)$. When $k=4$ and $k=4$, $v$ lose at most $\frac{4}{4}+\frac{4}{4}=2$. Now consider the case when $|F_{4^+}(v)|=3$. Then $n\cup N(v)$ consists of  two $1$-fans and a $6$-fan, $2$-fan and $1$-fan and a $5$-fan, $3$- fan and $1$-fan and a $4$-fan, two $2$-fans and a $4$-fan, or two $3$-fans and a $2$-fan. Similarly as for $|F_{4^+}(v)|=4$, $v$ lose at most 2 except for the case that $n\cup N(v)$ consists of two $3$-fans and a $2$-fan. Since $G$ avoids the configuration $(g)$, $|N_3(v)|\leq4$, and thus $v$ lose at most $2$. Thus we have the result.   
\end{proof}

\begin{fact}
    \label{fact9}
    Every $12^+$-face ends with a non-negative charge.
\end{fact}

\begin{proof}
 Let $f$ be a face of a planar graph such that $l(f)\geq 12$.  Note that $f$ lose charge to $l(f)\cap S_{5^-}$ by $(R1)-(R7)$. Since $G$ avoid the configuration $(c)$, $V(f)\cap S_{5^-}\leq \lfloor\frac{l(f)}{2}\rfloor$. $f$ lose highest charge to $l(f)\cap S_3$ by $(R2)$ when $|F(v)|=1$ for $v\in S_3$. So, it is sufficient to consider the worst-case scenario where  $v\in S_3$ for all $v\in l(f)\cap S_{5^-}$ and $|F(v)|=1$. We will find $l(f)$ that satisfies the following equation.
$2*l(f)-6-3*\lfloor\frac{l(f)}{2}\rfloor\geq 0$. It follows that for $l(f)$ is even, $l(f)\geq 12$, and for $l(f)$ is odd, $l(f)\geq 9$. Thus, we have the result.   
\end{proof}

\begin{fact}
    \label{fact10}
    If $|F(v)|=1$ for $v\in V(G)$ with $d(v)=3$, then $l(f)\geq 15$, where $f=F(v)$.
\end{fact}

\begin{proof}
    Let $v\in V(G)$ with $d(v)=3$ such that $|F(v)|=1$. Then, $G-\{v\}$ has three components. Since $G$ is a simple planar graph and $N(v)\in S_{7^+} $ as $G$ avoids the configuration $(c)$, each component of $G-\{v\}$ has face degree at least 3. Thus, in $G$ with the 6 edges contributed by $(v, u_i)$, where $u_i\in N(v)$, $l(f)\geq 3*3+6=15$. 
\end{proof}

\begin{fact}
    \label{fact11}
    Every  face $f$ with $5\leq l(f)\leq 11$ ends with a non-negative charge.
\end{fact}

\begin{proof}
  Let $f$ is a face in $G$ such that $5\leq l(f)\leq 11$. Then for all $v\in l(f)\cap S_3$,  $|F(v)|\geq 2$. Otherwise, $l(f)\geq 15$ by Fact 10; a contradiction. Thus each $v\in l(f)\cap S_{5^-}$ take at most $2$ from $f$ by $(R1)-(R7)$. We use the same equation used in Fact \ref{fact 9} to prove this. Since $G$ avoid the configuration $(c)$, $V(f)\cap S_{5^-}\leq \lfloor\frac{l(f)}{2}\rfloor$. $f$ lose at most $1$ to each of $v\in l(f)\cap S_{5^-}$. Thus, $v$ lose at most
$2*l(f)-6-2*\lfloor\frac{l(f)}{2}\rfloor$ and $2*l(f)-6-2*\lfloor\frac{l(f)}{2}\rfloor\geq 0$  for $l(f)\geq 6$, and  for $l(f)\geq 5$, when $l(f)$ is even and odd respectively. Thus, we have the result.   
\end{proof}

\begin{fact}
    \label{fact12}
    Every $4$-face, $f$, ends with a non-negative charge.
\end{fact}
\begin{proof}
 Let $l(f)=4$. $f$ starts with charge $2*4-6=2$, and want to show $f$ lose at most 2.  Since $G$ avoid the configuration $(c)$, $|l(f)\cap S_{5^-}|\leq 2$. If all $v\in l(f)\cap S_{5^-}$ are $4$ or $5$-vertices, i.e. $v\in l(f)\cap (S_4 \cup S_5)$, thus each of such a vertex $v$ take at most $1$ from $f$ by $(R1)-(R7)$. Thus, $f$ loses at most $2$ and we are done. When $v\in l(f)\cap S_3$, let $u\in N(v)$ such that $(v,u)$ $E_{4,3}$, where $4$-face incident with $(v,u)$ is the face $f$. Here we have to consider two cases:
 
 \textbf{Case 1:} $d(u)=8, d(u)\neq 9$, or $d(u)=9$ and $u$ does not fall into cases $(a)$ or $(b)$ in $(R3)$, then $v$ draw at most $1$ from $f$. Thus, $f$ loses at most 2 and we are done. 
 
 \textbf{Case 2:} $d(u)= 9$ and $u$ falls into cases $(a)$ and $(b)$ of $(R3)$. Note that in this case, the $d(w)\geq 8$ where $w\in N(u)$ such that $\{v,u,w\}$ on $f$. Thus, $|V(f)\cap S_3|=1$. So. $f$ lose at most $1\frac{1}{2}$ to $v\in l(f)\cap S_3$ by $(R1)$, and $(R3)$. Thus, we have the result.
    
\end{proof}

\begin{fact}
    \label{fact13}
    Every $3$-face, $f$, ends with a non-negative charge.
\end{fact}

\begin{proof}
     Let $l(f)=3$. $f$ starts with charge $2*3-6=0$. $f$ does not lose charge by $(R1)-(R7)$, thus remain non-negative.
\end{proof}

\textbf{Proof of Theorem \ref{maintheorem}}
\begin{proof}
 Consider a connected planar graph $G$ with minimum degree $\delta(G)\geq 3$ that does not contain a configuration as described in Theorem \ref{maintheorem}. Initially, assign a charge of $d(v)-6$ to each vertex and a charge of $2l(f)-6$ to each face. The initial total charge is $-12$, as per the vertex charging equation of Proposition \ref{dischargingequations}. Subsequently, reassign charges following rules $(R1)-(R7)$. We have previously demonstrated that every vertex ends with a non-negative charge (Facts 2, 6, 7, and 8), and similarly, every face ends with a non-negative charge (Facts 9, 11, 12, and 13).

However, this leads to a contradiction, as we initiated the graph with a total charge of $-12$ and concluded with a non-negative charge after the redistribution. Therefore, a planar graph with $\delta(G)\geq3$ must have at least one of the configurations $(a)-(h)$.

\end{proof}

\section{Preliminaries to prove Theorem \ref{IBNconfigurationproof}}

First, we will state an important definition and a lemma. Let $D\subset V(G)$. We say an edge set $E$ is an attachment of $D$ if for every $e\in E$ has at least one endpoint in $D$. Now we will prove the following lemma.
\begin{lemma}
\label{lem2}
Let $v\in V(G)$ be such that $d(v)\geq 8$, $E$ be an attachment of $N(v)$, and $G'=G-E$. 
If there exist a minimum independent dominating set $I'$ of $G'$ such that $|D|= s$ for $s\geq 2$, where $D=I'\cap N(v)$, and $N_{G'}(w_i)=\emptyset$    for some $w_i\in D$, and   $\sum_{j} |N_{G'}(w_j)\setminus N[v]| \leq s-2 $  for $w_j\in D\cap N_{G'}(v)$, then $|I'|>|I|$, where $I' $ is a minimum independent dominating set of $G$.

\end{lemma}

\begin{proof}
 Let $v\in V(G)$ be such that $d(v)\geq 8$, $E$ be an attachment of $N(v)$, and $G'=G-E$. 
Let $I'$  be a minimum independent dominating set of $G'$ such that  $|D|= s$ for $s\geq 2$, where $D=I'\cap N(v)$, and $N_{G'}(w_i)=\emptyset$ for some $w_i\in D$, and $\sum_{i} |N_{G'}(w_j)-N(v)| \leq s-2 $ for $w_j\in D\cap N_{G'}(v)$. Without loss of generality, take $N_{G'}(w_1)=\emptyset$, where $w_1\in D$. We want to show $|I'|>|I|$, where $I, I' $ are minimum independent dominating sets of $G$ and $G'$ respectively. Suppose not, i.e. $|I'|\leq|I|$. Note that  $w_1\in I'$ as $d_{G'}(w_1)=0$. Thus, $N(w_1)\cap I'=\emptyset$, otherwise, $I'\setminus \{w_1\}$ is a independent dominating set of $G$ of size $|I'|-1$, which is a contradiction since $|I'|\leq |I|\leq |I'|-1<|I'|$. Thus, $v\not \in I'$. If for every $v_i\in N_{G'}(w_i)-N[v]$, if $N(v_i)\cap (I'-D)\neq \emptyset$, then $(I'-D)\cup \{v\}$ is an independent set of $G$ with size $|I'|-(s-2)+1=|I'|-s+3 < |I'|$ which is a contradiction since $|I'|\leq |I|\leq |I'|-1<|I'|$. Thus, $|I'|>|I|$.

Note that if for every $v_i\in N_{G'}(w_i)-N[v]$, if $N(v_i)\cap (I'-D)=\emptyset$, then $\bigcup_{i} v_i\cup (I'-D)\cup \{v\}$ is an independent set of $G$ with size $|s-2|+|I'|-s+1=|I'|-1 < |I'|$ which is a contradiction since $|I'|\leq |I|\leq |I'|-1<|I'|$. Thus, $|I'|>|I|$.

\end{proof}

\begin{fact}
    \label{Fact14}
    Let $u\in D$ such that $d_{G'}(u)=0$, then $N_G(u)\cap I'=\emptyset$
\end{fact}
\begin{proof} 
  Let $u\in I'$ such that $d_{G'}(u)=0$. Suppose to the contrary that $N_G(u)\cap I'\neq \emptyset$. Then, $I'\setminus\{u\}$ is an independent dominating set of $G$ of size $|I'|-1$ which is a contradiction since $|I'|\leq|I|\leq |I'|-1<|I'|$.   
\end{proof}

\begin{fact}
    \label{fact15}
     Let $u\in I'$ such that for every $w\in N_{G'}(u)$, $N_G(w)\cap (I'\setminus \{u\})\neq \emptyset $ and $(u,z)\notin E(G')$ for some $z\in N_G(u)$, then $z \not \in I'$.
\end{fact}

\begin{proof}
  Let $u\in I'$ such that for every $w\in N_{G'}(u)$, $N_G(w)\cap (I'\setminus \{u\})\neq \emptyset $ and $(u,z)\notin E(G')$ for some $z\in N_G(u)$. Now, $I'\setminus \{u\}$ is an independent dominating set of $G$ of size $|I'|-1$ which is a contradiction since $|I'|\leq |I'|-1<|I'|$.   
\end{proof}

\section{Proof of Theorem \ref{IBNconfigurationproof}}
Here we prove the Theorem $\ref{IBNconfigurationproof}$

\begin{proof}

By Theorem \ref{PriddyWei}, if a planar graph $G$ with $\delta(G)\geq 3$ contains configurations $(a),(b),$ or $(c)$ from Theorem \ref{maintheorem}, then $b_i(G)\leq 8.$ Hence, we need to consider planar graphs with $\delta(G)\geq 3$ containing the configurations $(d), (e), (f), (g),$ and $(h)$.

Now, we use the lemma \ref{lem2} and Fact \ref{Fact14} to prove that if $G$ has at least one from configurations $(d)-(h),$ then the independent bondage number of $G$ is at most 8. To prove this, we only need to find a good attachment of $N(v)$ such that $|E|\leq8$. An attachment is good if $\gamma_i(G-E)>\gamma_i(G)$.

Consider $(d)(i)$ with $d(v)=10$ (figure \ref{nineconfig} $(d)$), $d(u_i)=3$ for $i=1,3,5,7$, and $d(u_9)\leq 5$ such that $\{ x,y \} (\textit{if exists}) \subseteq N(u_9) \setminus \{u_8, u_9, v\}$, and \\
$E=\{ u_1v, u_1u_2, u_1u_{10}, u_3u_4, u_5u_6, u_7u_8, u_9x, u_9y\}$. $d_{G'}(u_1)=0$. Note that since $u_1\in I'$, $\{v,u_2,u_{10}\}\not\subset I'$ by Fact \ref{Fact14}. Thus, $u_3\in I'$ as $(u_3,u_4)\not \in E(G')$. $u_4\not \in I'$ by Fact \ref{fact15} as $u_3\in I'$ such that $u_2 \in N(u_3)$ and $v\in N(u_3)$ with $u_1\in N(u_2), N(v)$ and $u_1\in I'$. So, $u_5\in I'$ as $(u_5, u_6) \not\in E(G')$ and $\{v, u_4\} \not\subset I'$. By a similar argument, $\{u_7,u_9\} \subset I'$. Thus, $D=\{u_1,u_3,u_5,u_7,u_9\}$ and thus   $s=5$.  $\sum_{i}^{s}(d_{G'}(w_i)-|D_i|)=0 < s-2$  for $w_i \in D$. Thus, this case satisfies the conditions of Lemma \ref{lem2}, and thus $|I'|> |I|$. For $(d)(ii)$ with $d(v)=10$, $d(u_i)=3$ for $i=1,3,5$, and $d(u_i)\leq 4$ for $i=7,9$ such that $\{x\}\subseteq N(u_7)\setminus \{u_6,u_8,v\}$ and $\{y\}\subseteq N(u_9)\setminus \{u_8,u_{10},v\}$, and $E=\{u_1v,u_1u_2,u_1u_{10},u_3u_4,u_5u_6,,u_7x, u_7u_8,u_9y,u_9u_{10}\}$. Here also we get $d_{G'}(u_1)=0$,$s=|D|=|\{u_1,u_3,u_5,u_7,u_9\}|=5$, $\sum_{i}^{s}(d_{G'}(w_i)-|D_i|)=0 < s-2$  for $w_i \in D$.

For $(d)(i)$ with $d(v)=9$, $d(u_i)=3$ for $i=1,3,5$, and $d(u_7)\leq 5$ such that $\{x,y\}\subseteq N(u_7)\setminus \{u_6,u_8,v\}$, and $E=\{u_1v,u_1u_2,u_1u_{9},u_3u_4,u_5u_6,u_7x,u_7y,u_7u_8\}$. $d_{G'}(u_1)=0$, $s=|D|=|\{u_1,u_3,u_5,u_7\}|=4$, $\sum_{i}^{s}(d_{G'}(w_i)-|D_i|)=0 < s-2$  for $w_i \in D$. Thus, this case satisfies the conditions of Lemma \ref{lem2}, and thus $|I'|> |I|$. So, if $G$ has the configuration $(d)(i)$, $b_i(G)\leq 8$. For $(d)(ii)$ with $d(v)=9$, $d(u_i)=3$ for $i=1,3$, and $d(u_i)\leq 4$ for $i=5,7$ such that $\{x\}\subseteq N(u_5)\setminus \{u_4,u_6,v\}$ and $\{y\}\subseteq N(u_7)\setminus \{u_6,u_{8},v\}$, and $ E=\{ u_1v, u_1u_2, u_1u_{9}, u_3u_4, u_5x, u_5u_6, u_7y, u_7u_{8} \}.$ $d_{G'}(u_1)=0$, $s=|D|=|\{u_1,u_3,u_5,u_7\}|=4$, $\sum_{i}^{s}(d_{G'}(w_i)-|D_i|)=0 < s-2$  for $w_i \in D$. Thus, this case satisfies the conditions of Lemma \ref{lem2}, and thus $|I'|> |I|$. So, if $G$ has the configuration $((d)(ii)$, $b_i(G)\leq 8$.

\vspace{0.3cm}

Now consider the configuration $(e)(i)(\alpha)$ (figure \ref{nineconfig} $(e)(i)$). Let  $d(u_i)=3$ for $i=2,4,6$, and $d(u_8)\leq 5$ such that $\{x,y\}\subseteq N(u_8)\setminus \{u_7,u_8,v\}$, and \\$E=\{u_1u_2,u_2v, u_2u_{3},u_4u_5, u_6u_7, u_8x, u_8y,u_8u_9\}$. $d_{G'}(u_2)=0$, $s=|D|=|\{u_2, u_4, u_6, u_8\}|=4$, $\sum_{i}^{s}(d_{G'}(w_i)-|D_i|)=0 < s-2$  for $w_i \in D$. Thus, this case satisfies the conditions of Lemma \ref{lem2}, 
thus $|I'|> |I|$. Now consider the configuration $(e)(i)(\beta)$. Let  $d(u_i)=3$ for $i=1,3,5,7$, and $d(u_9)\leq 7$ such that $\{x_1,x_2,x_3,x_4\}\subseteq N(u_9\setminus \{u_8,y,v\}$, where $y\in N(u_9)$ such that $u_9y$ 
on $f_1$, and $E=\{u_1u_2,u_1v,u_1y, u_3u_{4},u_5u_6,u_7u_8,u_9x_1,u_9y\}$. $d_{G'}(u_1)=0$, $s=|D|=|\{u_1, u_3, u_5, u_7\}|=4$ or $s=|D|=|\{u_1, u_3, u_5, u_7, u_9\}|=5$ (if 
($N(u_9)\setminus \{u_8,v,y,x_1\})\cap I' = \emptyset$). $\sum_{i}^{s}(d_{G'}(w_i)-|D_i|)=3$  for $w_i \in D$ (this case happens only if $u_9\in I'$, thus when $s=5$). So, in this case also, $\sum_{i}^{s}(d_{G'}(w_i)-|D_i|)\leq s-2$. Thus, 
this case satisfies the conditions of Lemma \ref{lem2}, and thus $|I'|> |I|$. Now consider the configuration $(e)(i)(\delta)$. Let  $d(u_i)=3$ for $i=1,3,5,8$, and $E=\{u_1u_2,u_1v,u_1y, u_3u_{4},u_5u_6,u_7u_8,u_8v,u_8u_9\}$, where $y\in N(u_1)$ 
such that $u_1y$ on $f_1$. $d_{G'}(u_1)=0$, $s=4$, $\sum_{i}^{s}(d_{G'}(w_i)-|D_i|)=0 < s-2$  for $w_i \in D$. 
Thus, this case satisfies the conditions of Lemma \ref{lem2}, and thus $|I'|> |I|$. So, if $G$ has the configuration 
$(e)(i)$, $b_i(G)\leq 8$.

\vspace{0.3cm}

Now consider the configuration $(e)(ii)$. Let  $d(u_i)=3$ for $i=1,3,5,7$, and $d(u_9)\leq 6$ such that $\{x_1,x_2,x_3\}\subseteq N(u_9\setminus \{v,y,z\}$, where $\{y,z\}\subseteq N(u_9)$ such that $u_9y$ on $f_1$ and $u_9z$ on $f_2$, and $E=\{u_1u_2,u_1v,u_1y, u_3u_{4},u_5u_6,u_7u_8,u_9y,u_9z\}$. $d_{G'}(u_1)=0$, $s=|D|=|\{u_1, u_3,u_5,u_7\}|=4$ or $s=5$ (if ($N(u_9)\setminus \{z,y,v\})\cap I' = \emptyset$, $\{u_1, u_3,u_5,u_7, u_9\}$). $\sum_{i}^{s}(d_{G'}(w_i)-|D_i|)=|\{ x_1, x_2, x_3\}|=3$  for $w_i \in D$. This case happens only if $u_9\in I'$, thus when $s=5$. So, $\sum_{i}^{s}(d_{G'}(w_i)-|D_i|)=2\leq 5-3$. Thus, $(e)(ii)$ satisfies the conditions of Lemma \ref{lem2}, and thus $|I'|> |I|$. So, $b_i(G)\leq 8$.

\vspace{0.3cm}

Now consider the configuration $(f)$. Let  $d(u_i)=3$ for $i=1,2,4,5$, and $d(u_7)\leq 5$ such that $\{x_1,x_2\}\subseteq N(u_7)\setminus \{u_6,u_8,v\}$, and \\ $E=\{u_2u_3,u_2v, u_2z,u_4y,u_5u_6,u_7x_1,u_7x_2,u_7u_8\}$. This forces $u_1\in I'$. $d_{G'}(u_2)=0$. Note that $\{u_1, u_2,u_4,u_5, u_7\}\subset I'$. Thus,  $s=5$. $\sum_{i}^{s}(d_{G'}(w_i)-|D_i|)=|$\{y,z\}$|=2 <5-2$  for $w_i \in D$. Thus, this case satisfies the conditions of Lemma \ref{lem2}, and thus $|I'|> |I|$. So, if $G$ has the configuration $(f)$, $b_i(G)\leq 8$.

\vspace{0.3cm}

Now consider the configuration $(g)$. Let  $d(u_i)=3$ for $i=2,4,5,7,8$, and \\$E=\{u_2u_3,u_2v, u_2x,u_4y,u_5u_6,u_7z,u_1u_8\}$. This has $|E|=7<8$. $d_{G'}(u_2)=0$, $s=5$ by a similar argument like for configuration $(f)$, we have to write a face for this saying since we deleted edges, it forces some vertices to be in the independent dominating set. $\sum_{i}^{s}(d_{G'}(w_i)-|D_i|)2 <s-2$  for $w_i \in D$ (also I can make this set to be 1 by deleting one more edge from $\{u_8z,u_5y\}$). Thus, this case satisfies the conditions of Lemma \ref{lem2}, and thus $|I'|> |I|$. So, if $G$ has the configuration $(g)$, $b_i(G)\leq 8$.

\vspace{0.3cm}
Now consider the configuration $(h)$. This is similar to $(e)(i)(\alpha)$. Let  $d(u_i)=3$ for $i=1,3,5,7,9$, and  $E=\{u_1u_2,u_1v, u_1x,u_3u_4,u_5u_6,u_7u_8,u_9u_{10}\}$. Here, $|E|=7<8$. $d_{G'}(u_1)=0$, $s=5$, $\sum_{i}^{s}(d_{G'}(w_i)-|D_i|)=0<s-2$  for $w_i \in D$. Thus, this case satisfies the conditions of Lemma \ref{lem2}, and thus $|I'|> |I|$. So, if $G$ has the configuration $(h)$, $b_i(G)\leq 8$.

\end{proof}

\end{document}